\documentclass[12pt]{amsart}
\newtheorem{theorem}{Theorem}[section]
\newtheorem{lemma}[theorem]{Lemma}
\newtheorem{proposition}[theorem]{Proposition}

\newtheorem{definition}[theorem]{Definition}

\newtheorem{remark}[theorem]{Remark}
\newtheorem{remarks}[theorem]{Remarks}

\newcommand{\bC}{\mathbb{C}}

\newcommand{\bQ}{\mathbb{Q}}
\newcommand{\bR}{\mathbb{R}}
\newcommand{\bZ}{\mathbb{Z}}

\newcommand{\bN}{\mathbb{N}}

\newcommand{\bI}{\mathbb{I}}

\newcommand{\GL}{\mathrm{GL}}

\newcommand{\disp}{\displaystyle}

\newcommand{\Adj}{\mathrm{Adj}}
\def\disp{\displaystyle}
\input{cyracc.def}

 \numberwithin{equation}{section}
\begin{document}
\baselineskip=17pt%20pt
\title[An arithmetic study of the formal Laplace transform in several variables.]
{An arithmetic study of the formal Laplace transform in several
variables.}
\author{Said Manjra}
\curraddr{Department of Mathematics \\
Faculty of science\\
Al-Imam University \\
P.O.Box: 240337. Riyadh 11322. Saudi Arabia.}
\email{smanjra@uottawa.ca} \subjclass[2000]{12H25}
\thanks{}
\dedicatory{}

\maketitle

\begin{abstract}
Let $K$ be a number field, and let $K(x_1,...,x_d)$ be the field of
rational fractions in the variables $x_1,...,x_d$. In this paper, we
introduce two kinds of Laplace transform adapted to solutions of the
differential $K(x_1,...,x_d)$-modules with regular singularities,
and give some of their basic differential and arithmetic properties.
The purpose of this article is to provide some tools which might be
useful, in particular, for the arithmetic study of the differential
$K(x_1,...,x_d)$-modules associated to $E$-functions in several
variables.

\end{abstract}

\section{Introduction}

A power series $\sum_{n\ge 0}a_nx^n\in \bC[[x]]$ is said to be \emph{Gevrey of order s} if there exists a positive constant $C$ such $|a_n|< C^nn!^s$ for all $n>0$. It is known since the beginning of the last century that any formal power series  arising in the asymptotic expansion of any solution of a linear or non-linear analytic differential equation is Gevrey of rational order $s$ (\cite{G}, \cite{Ma} and \cite{P}).  Moreover, the Gevrey series occurring in Taylor or asymptotic expansions of classical special functions have  particular arithmetic properties which can be summarized in the concept of the \emph{Gevrey series of arithmetic type} (\cite{An}, \cite{A3}):

A power series $\sum_{n\ge 0}a_nx^n$ is said to be \emph{arithmetic Gevrey series} of order $s\in {\bQ}$ if its coefficients
$a_n$ are algebraic  numbers and if there exists a constant
$C>0$ such that the absolute value of the conjugates of the algebraic number $a_n/(n!)^s$ is less than $C^n$, and that, for all
$n\in {\bN}$, the common denominator  $d_n$ of the numbers $a_0=a_0/(0!)^s, \ldots, a_n/(n!)^s$  is less than  $C^n$. For instance, the Taylor expansion at the origin of the Airy function is arithmetic Gevrey series of precise order $-2/3$.
This  class includes the confluent and non confluent generalized hypergeometric series with rational parameters,
the \emph{Barnes} generalized hypergeometric series $_pF_{q-1}$ with rational parameters, any series which is algebraic
over $\overline{\bQ}(x)$, and two especially  well-known series: $G$\emph{-functions} ($s=0$) and  $E$\emph{-functions} ($s=-1$) which have proved
useful in number theory and have applications in transcendence proofs and differential equations (e.g. see \cite{Be}, \cite{B} and \cite{L}).

The theory  of $G$-functions (the case of arithmetic Gevrey series of order 0) in one-variable is now well known thanks to Bombieri, Chudnovsky, Dwork,  Andr\'e and others  who, furthermore,  brought to light its connections with arithmetic geometry (e.g. see \cite{A0}, \cite{Ch}, \cite{D} and \cite{DGS}).

 A theory  of  arithmetic Gevrey series of order $s\ne 0$ in one-variable was recently developed by Andr\'e  in which a \emph{Laplace transform} is used in an essential way \cite{An}. Based on a theorem of Chudnovsky \cite{DGS},  he gave interesting structure results for the $\bC[x]$-differential equation of minimal order annihilating an arithmetic Gevrey series of precise order $s\ne 0$. As an  application, these results enabled him to deduce the fundamental theorem of the Siegel-Shidlovskii theory on the algebraic independence of values of $E$-functions at algebraic points, and the Lindemann-Weierstrass theorem.

A  theory of $G$-functions in several variables was established in full generality by Andr\'e, Baldassarri and Di Vizio  (\cite{Ab}, \cite{AB} and \cite{D}), but a more general theory of arithmetic Gevrey series in several variables has yet to mature. Given the work of Andr\'e, it is natural that an appropriate notion of Laplace transform in several variables will be needed to build such theory.

In the present article, we
introduce two kinds of  Laplace transform
 which extend the ones introduced respectively in
 (\cite{An}, \cite{Mn}) and \cite{MR} to fundamental solution matrices of
integrable systems of partial differential equations on
$K((x_1,\ldots,x_d))$ with regular singularities at the origin ($K$ being a number field). The
general form of  such fundamental solution matrices is
$Yx_1^{\Lambda_1}\ldots x_d^{\Lambda_d}$, where
$\Lambda_1,\ldots,\Lambda_d$ are square matrices with entries in the
algebraic closure $\overline{K}$ of $K$, and where  $Y$ is an
invertible matrix with entries in $\overline{K}((x_1,\ldots,x_d))$
\cite{GL}. The new Laplace transforms studied in this paper are:
\begin{itemize}
 \item[(1)]
 the {\em standard Laplace transform}, which applies
to the entries of $Yx_1^{\Lambda_1}\ldots x_d^{\Lambda_d}$,
\item[(2)] the  {\em formal Laplace transform},
which applies directly to the solution matrix  $Yx_1^{\Lambda_1}\ldots
x_d^{\Lambda_d}$.
\end{itemize}
These transformations are given under the assumption that all the
eigenvalues of $\Lambda_1,\ldots,\\ \Lambda_d$ are non-integral complex
numbers, and have properties of commutations with the derivations
$\partial/\partial x_i$ ($i=1,\ldots,d$) which extend those in
one-variable case (\eqref{op2} and \eqref{op3}). Hence, they
preserve the classical duality between the Laplace transform and the
Fourier-Laplace transform. Moreover, for any
$\underline{\tau}=(\tau_1,\ldots,\tau_d)\in (K\setminus\{0\})^d$,
the formal transformation is adapted to have a duality with the
generalized Fourier-Laplace transform
$\mathcal{F}_{\underline{\tau}}$ with respect to $\underline{\tau}$
(formula \eqref{op3}), defined as the $K$-automorphism of
$K[x_1,\ldots,x_d,\partial/\partial x_1,\ldots,\partial/\partial
x_d]$ determined by:
$$ \mathcal{F}_{\underline{\tau}}(x_i)= -\frac{1}{\tau_i}\frac{\partial}{\partial x_i},
\quad  \mathcal{F}_{\underline{\tau}}(\frac{\partial}{\partial x_i})
=\tau_i x_i,\quad (i=1,\ldots,d). $$
 For
$\underline{\tau}=(1,\ldots,1)$, $\mathcal{F}_{\underline{\tau}}$ is
just the classical Fourier-Laplace transform. In this case, we write
$\mathcal{F}$ instead of $\mathcal{F}_{\underline{\tau}}$.

The difference between the two transformations is that the standard
 one involves transcendental numbers and applies to terms with logarithms \eqref{tr3}, while the
formal one is defined independently  of such numbers and does not
apply directly to terms with logarithms \eqref{trf2}.

If $v$ is a finite place of $K$ above a prime number $p(v)$, we
prove some arithmetic properties of the standard (resp. formal)
Laplace transform in the case where all the eigenvalues of
$\Lambda_1,\ldots,\Lambda_d$ belong to $K\cap\bZ_{p(v)}\setminus\bZ$
(resp. $\bQ\cap\bZ_{p(v)}\setminus\bZ $) (Propositions \ref{pro1}
and \ref{pro3}).

As an application, we show in \S3.4  how the standard Laplace
transform acts on the arithmetic Gevrey series
 in several variables (Proposition \ref{gev}).

We think that this paper provides some tools which might be useful for the development of an arithmetic theory of differential equations in higher dimension, and in particular, for the arithmetic study of the differential
$K(x_1,\ldots,x_d)$-module generated by the different derivatives of
a Gevrey series of nonzero order in several variables.

The aim of this theory is to develop   techniques which allow, in particular,  to obtain results on algebraic independence of values of  the exponential function in several variables,  and  of the functions of the form: $$f(x,y)=P(e^x,\, _2F_1(a,b,c;y)), \quad \text{with} \quad a,b,c\in \bQ,$$
where $P$ is a polynomial. The values of such function are related to $e$ and $\pi$.
\section{ Notation}
Let $K$ be a number field and let $\Sigma_{\rm f}$ be the set of all
finite places $v$ of $K$. For each $v\in \Sigma_{\rm f}$ above a
prime number $p=p(v)$, we normalize the corresponding $v$-adic
absolute value so that $|p|_{v}=p^{-1}$ and we put
$\pi_{v}=|p|_{v}^{1/(p(v)-1)}$. We also fix an embedding
$K\hookrightarrow {\bC}$.

Let $K(x_1,\ldots,x_d)$ be the field
 of rational functions in the variables $x_1,\ldots,x_d$ with coefficients in $K$, with  $d\in {\bZ}_{> 0}$.
 Put $\underline{x}=(x_1,\ldots,x_d)$, $\underline{\frac{1}{x}}=(\frac{1}{x_1},\ldots, \frac{1}{x_d})$, $
\partial_i=\frac{\partial}{\partial x_i}$, for $i=1,\ldots,d$, $\underline{1}=(1,\ldots,1)$, and
$\underline{\alpha}=(\alpha_1,\ldots,\alpha_d)$ for  the elements of
${\bN}^d$. For
 $\underline{\alpha},\underline{\beta}\in {\bN}^d$, $\gamma\in K$ and $n\in \bZ_{>0}$, we set :
$(\gamma)_0=\gamma$, $(\gamma)_{n+1}=\gamma(\gamma+1)\ldots(\gamma+n)$,
 $$
|\underline{\alpha}|=\sum_{1\le i\le
d}\alpha_i,\quad\underline{\alpha}!=\prod_{1\le i\le
 d}\alpha_i!,\quad
\underline{x}^{\underline{\alpha}}=x_1^{\alpha_1}\ldots
x_d^{\alpha_d},\quad
\underline{\partial}^{\underline{\alpha}}=\partial_1^{\alpha_1}\ldots
\partial_d^{\alpha_d},$$
$$\underline{\alpha}\le
\underline{\beta}\quad\Longleftrightarrow\;\;\alpha_i\le
 \beta_i\quad \text{for all}\quad i=1,\ldots,d,\quad\text{and}\quad
 \Big({\underline{\alpha}\atop\underline{\beta}}\Big)=\prod_{1\le i\le d}\Big({\alpha_i\atop \beta_i}\Big)
 \quad \text{for }\quad\underline{\alpha}\ge
\underline{\beta},$$
$$\underline{\alpha}<
\underline{\beta}\;\;\Longleftrightarrow\;\;\underline{\alpha}\le
\underline{\beta}\quad \text{ and}\;\; \alpha_i<
 \beta_i\quad \text{for some}\quad i=1,\ldots,d. $$ For a power series
$f=\sum_{\underline{\alpha}\in {\bN}^d}
a_{\underline{\alpha}}\underline{x}^{\underline{\alpha}}\in
K[[\underline{x}]]$, we denote
$f(\frac{1}{\underline{x}})=\sum_{\underline{\alpha}\in {\bN}^d}
a_{\underline{\alpha}}\underline{x}^{-\underline{\alpha}}\in
K[[\frac{1}{\underline{x}}]]$. If $v$ is a finite place
of $\Sigma_{\rm f}$ and  $s\in \bQ$ is a rational number, we set
\begin{equation*}
\begin{aligned}
 {\mathcal R}_v^{s}(f)=&\{y\in K[[\underline{x}]]\;|\; r_{v}(y)\ge
r_{v}(f)\;\pi_{v}^{s}\},
\end{aligned}
\end{equation*}
where $r_{v}(f)$ and $r_{v}(y)$ denote respectively the radius of convergence of $f$ and $y$ with respect to $v$.
$\underline{\gamma}=(\gamma_1,\ldots,\gamma_d)$
and $\underline{k}=(k_1,\ldots,k_d)$ will denote  respectively
multi-exponents in $K^d$ and $\bN^d$. Also,  we will use the
following notation
$$(\log \underline{x})^{\underline{k}}=(\log x_1)^{k_1}\ldots(\log
x_d)^{k_d},\quad
h_{\underline{\gamma},\underline{k}}=\underline{x}^{\underline{\gamma}}(\log
\underline{x})^{\underline{k}},$$
$$(\underline{\gamma})_{\underline{0}}=\prod_{1\le i\le d}\gamma_i, \quad \text{and} \quad (\underline{\gamma})_{\underline{\alpha}}=(\gamma_1)_{\alpha_1}\ldots(\gamma_d)_{\alpha_d}.$$
\section{Standard Laplace transform }
\subsection{3.1 Review of the one-variable case}
In \cite{An}, Andr\'e extended the definition of the classical
Laplace transform $\mathcal{L}$, in the formal way, to logarithmic
solutions in one-variable. Later, in \cite{Mn}, the present author
gave some arithmetic properties of this $\mathcal{L}$. In this
paragraph, we recall the definition of the Laplace transform
$\mathcal{L}(h_{\gamma,k})$ of the term $h_{\gamma,k}:=x^\gamma(\log
x)^k$ where $(\gamma,k)\in
 K\setminus{\bZ}_{\le 0}\times{\bN}$, and its differential and arithmetic
properties \cite[4]{Mn} \footnote{In the case  where $(\gamma,k)\in
 {\bZ}_{< 0}\times{\bN}$, the Laplace transform
$\mathcal{L}$ is defined differently \cite[(4.13)]{Mn}, and does not
satisfy the second formula of \eqref{op1} below when $\gamma\in
 {\bZ}_{< 0}$ and $k=0$ \cite[(4.11)]{Mn}. Hence, the definition of $\mathcal{L}$
 in the case where $(\gamma,k)\in
 {\bZ}_{< 0}\times{\bN}$ will not be interesting in our context.}.

Fix an embedding $K\hookrightarrow\bC$. Let $\gamma$ be an element
of $K\setminus\bZ_{\le 0}$ such that $\Re e(\gamma)>-1$, $k$ a nonnegative integer,
 and let $h_{\gamma,k}$ denote the function
defined by $h_{\gamma,k}(x)=x^{\gamma}(\log x)^{k}$; $x>0$. The
classical Laplace transform of $h_{\gamma,0}$, denoted ${\mathcal
L}(h_{\gamma,0})$, is defined by
\begin{equation*}
 {\mathcal L}(h_{\gamma,0})(x)=\int_0^\infty
e^{-xt}h_{\gamma,0}(t) dt=\int_0^\infty e^{-xt}t^\gamma
dt=\Gamma(\gamma+1)x^{-\gamma-1}.
\end{equation*}
The $k$-th derivative of this equation with respect to $\gamma$
gives
$$(\frac{d}{d\gamma})^{k}\Big(\Gamma(\gamma+1)x^{-\gamma-1}\Big)=
\int_{0}^{\infty}e^{-xt}t^{\gamma} (\log t)^{k}dt=
\int_{0}^{\infty}e^{-xt}h_{\gamma,k}(t)dt={\mathcal
L}(h_{\gamma,k})(x),$$ and by iteration of Leibniz formula  we find
\begin{equation}
\label{real}
 {\mathcal L}(h_{\gamma,k})(x)=
\sum_{j=0}^{k}\binom{k}{j}\Gamma^{(j)}(\gamma+1)x^{-\gamma-1}(-1)^{k-j}(\log
x)^{k-j},
\end{equation}
where, $\Gamma^{(j)}$ denotes the $j$-th derivative of the Euler's Gamma
function $\Gamma$ \cite[(4.2)]{Mn}. To extend the definition of the
Laplace transform ${\mathcal L}$ of $h_{\gamma,k}$ to any $\gamma\in
K\setminus\bZ_{\le 0}$, we have to introduce the functions
$F_{\gamma,k,n}$ \cite[(4.8)]{Mn} \cite[5]{DP}, defined for $n\in
\bZ_{>0}$ and $\gamma\in
K\setminus\bZ_{\le 0}$, by
$$F_{\gamma,k,n}(x)= \sum_{m=0}^{n}
\frac{(-1)^{m}}{m!(n-m)!}\frac{x^{\gamma+n+1}}{m+\gamma+1}
\sum_{\ell=0}^{k}\frac{k!(-1)^{k-\ell}}{\ell!(m+\gamma+1)^{k-\ell}}(\log
x)^{\ell}.$$ Andr\'e observed that the function $x^{n+1}{\mathcal
L}(F_{\gamma,k,n})$ is  independent of the choice of $n$ for $n\ge
-\Re e(\gamma)-1$ (cf. \cite[5.3.6]{An}). According to this remark,
the Laplace transform ${\mathcal L}$, just defined for $\gamma\in
K\setminus\bZ_{\le 0}$ with $\Re e(\gamma)>-1$ (see \eqref{real}
above), can be extended to any $\gamma\in K\setminus\bZ_{\le 0}$ by
putting
\begin{eqnarray}
{\mathcal L}(h_{\gamma,k})&=& x^{n+1}{\mathcal
L}(F_{\gamma,k,n}),\;\;\;\text{for}\;\;\; n\ge -\Re e(\gamma)-1.
\end{eqnarray}
With this definition,  we find formally \cite[(4.12)]{Mn}:
\begin{equation}
\label{tr1}
\begin{aligned}
{\mathcal L}(h_{\gamma,k})=&\; \Gamma(\gamma+1)\;
x^{-\gamma-1}\;\sum_{j=0}^{k}\;\rho^{(k)}_{\gamma,j}\;(\log x)^j\\
&
\text{with}\;\;\rho^{(k)}_{\gamma,k}=(-1)^{k},\;\;\rho^{(k)}_{\gamma,j}\in
<\Gamma(\gamma),\ldots,\Gamma^{(k)}(\gamma)>_{{\bQ}[\gamma]},\;\;j=0,\ldots,k-1.
\end{aligned}
\end{equation}
\ \\
Moreover, this transformation ${\mathcal L}$ has the following
formal properties (cf. \cite[(4.10), (4.11)]{Mn}, \cite[(5.3.7),
(5.3.8)]{An}):
\begin{equation}
\label{op1}
\begin{aligned}
\frac{d}{dx}({\mathcal L}(h_{\gamma,k}))=&{\mathcal
L}(-xh_{\gamma,k}),\;\;\text{and}\hspace{1cm} {\mathcal
L}\Big(\frac{d}{dx}(h_{\gamma,k})\Big)= x{\mathcal
 L}(h_{\gamma,k}).
\end{aligned}
\end{equation}
Combining formulas \eqref{tr1} and \eqref{op1}, we obtain the
following recurrence relation with respect to the index $j$,
\begin{equation}
\label{rh}
\begin{aligned}
\rho^{(k)}_{\gamma+1,j-1}=-\rho^{(k)}_{\gamma,j-1}+\disp\frac{j}{\gamma+1}\;\rho^{(k)}_{\gamma,j}.
\end{aligned}
\end{equation}
In the sequel, we fix $\gamma$ in $K\setminus\bZ$ and $k$ in $\bN$.
With the notation above, the following Lemma shows that, for any
$n\in \bN$ and any $0\le j\le k$, the coefficient
$\rho^{(k)}_{\gamma\pm n,j}$ can be written as a
$\bQ(\gamma)$-linear combination of
$\rho^{(k)}_{\gamma,j},\ldots,\rho^{(k)}_{\gamma,k}$.
\begin{lemma}
\label{rholem} For any $0\le j\le k$ and for any $n\in \bN$, there
exist numbers $r_{\gamma+n,j}^{(k,\ell)}$,
$r_{\gamma-n,j}^{(k,\ell)}$ in $\bQ(\gamma)$ for $\ell=j,\ldots,k$,
such that
\begin{equation}
\label{rho1}
\rho^{(k)}_{\gamma+n,j}=\sum_{\ell=j}^{k}\rho^{(k)}_{\gamma,\ell}\;
r_{\gamma+n,j}^{(k,\ell)},\quad\text{and}\quad
\rho^{(k)}_{\gamma-n,j}=\sum_{\ell=j}^{k}\rho^{(k)}_{\gamma,\ell}\;
r_{\gamma-n,j}^{(k,\ell)}.
\end{equation}
 Moreover, for any finite place $v$ in $\Sigma_{\rm f}$,
\begin{equation}
\label{enq}
\begin{aligned}
\limsup_{n\longrightarrow  +\infty} \Big|r_{\gamma\pm
n,j}^{(k,\ell)}\Big|_{v}^{1/n}\le 1.
\end{aligned}
\end{equation}
\end{lemma}
The proof is similar to that of  \cite[Lemma 4.2]{Mn}. We reproduce
it here in order to clarify the nature of the coefficients
$r_{\gamma\pm n,j}^{(k,\ell)}$. This will be useful in the proof of
Proposition \ref{gev}.
\begin{proof} We prove the lemma by
downward induction on the index $j$. In the case $j=k\ge 0$, by
\eqref{tr1}, it suffices to take $r_{\gamma\pm  n,k}^{(k,k)}=1$ for
any $n\in\mathbb{N}$. Assume now that the lemma is true for some
index $j$ with $1\le j\le k$. From \eqref{rh}
 and by iteration on $n\ge
1$, we  find
\begin{equation*}
\begin{aligned}
\rho^{(k)}_{\gamma+n,j-1}&=(-1)^n\rho^{(k)}_{\gamma,j-1}+j\sum_{i=0}^{n-1}\disp\frac{(-1)^{n+i+1}}{\gamma+i+1}\;
\rho^{(k)}_{\gamma+i,j}\\
&=(-1)^n\rho^{(k)}_{\gamma,j-1}+j\sum_{\ell=j}^{k}\rho^{(k)}_{\gamma,\ell}
\sum_{i=0}^{n-1}\disp\frac{(-1)^{n+i+1}}{\gamma+i+1}\;r_{\gamma+i,j}^{(k,\ell)},
\end{aligned}
\end{equation*}
and
\begin{equation*}
\begin{aligned}
\rho^{(k)}_{\gamma-1,j-1}&=-\rho^{(k)}_{\gamma,j-1}+\disp\frac{j}{\gamma}\;\rho^{(k)}_{\gamma-1,j},\\
\rho^{(k)}_{\gamma-n,j-1}&=(-1)^n\rho^{(k)}_{\gamma,j-1}+
j\sum_{i=1}^{n}\disp\frac{(-1)^{n-i}}{\gamma-i+1}\rho^{(k)}_{\gamma-i,j}\\
&=(-1)^n\rho^{(k)}_{\gamma,j-1}+
j\sum_{\ell=j}^k\rho^{(k)}_{\gamma,\ell}\sum_{i=1}^{n}\disp\frac{(-1)^{n-i}}{\gamma-i+1}r_{\gamma-i,j}^{(k,\ell)}.
\end{aligned}
\end{equation*}
Putting for all $j\le \ell\le k$ and all $n\ge 1$
\begin{equation}
\begin{aligned}
\label{down}
r_{\gamma\pm  n,j-1}^{(k,j-1)}&=(-1)^{n}\\
r_{\gamma+n,j-1}^{(k,\ell)}&=j\sum_{i=0}^{n-1}\disp\frac{(-1)^{n+i+1}}{\gamma+i+1}\;
 r_{\gamma+i,j}^{(k,\ell)}\\
r_{\gamma-n,j-1}^{(k,\ell)}&=j\sum_{i=1}^{n}\disp\frac{(-1)^{n-i}}{\gamma-i+1}\;r_{\gamma-i,j}^{(k,\ell)},
\end{aligned}
\end{equation}
 we get, $r_{\gamma\pm n,j-1}^{(k,\ell)}\in
{\bQ}(\gamma)$ for $\ell=j-1,\ldots,k$,
$$\rho^{(k)}_{\gamma+n,j-1}=\sum_{\ell=j-1}^{k}\rho^{(k)}_{\gamma,\ell}\;
 r_{\gamma+n,j-1}^{(k,\ell)}\quad \text{and}\quad \rho^{(k)}_{\gamma-n,j-1}=\sum_{\ell=j-1}^{k}\rho^{(k)}_{\gamma,\ell}\;
 r_{\gamma-n,j-1}^{(k,\ell)}.$$
This proves the first statement of the lemma. Let now $v\in
\Sigma_\text{f}$. By induction hypotheses, we have
$\disp\limsup_{n\longrightarrow\infty} \Big|r_{\gamma\pm
n,j}^{(k,\ell)}\Big|_{v}^{1/n}\le 1$ for $\ell=j,\ldots,k$. Since
$\gamma$ is an element of $K$, it is non-Liouville for $p(v)$ and
consequently we have
$\disp\limsup_{n\longrightarrow\infty}\Big|\disp\frac{1}{\gamma \pm
n}\Big|_{v}^{1/n}=1$
  (cf. \cite[VI.1.1]{DGS}). We deduce
$$\disp\limsup_{n\longrightarrow\infty}\Big(\max_{0\le i\le
n}|r_{\gamma\pm i,j}^{(k,\ell)}\Big|_{v}^{1/n}\Big)\le 1
,\;\;\;\ell=j,\ldots,k$$ and
$$\disp\limsup_{n\longrightarrow\infty} \Big(\max_{0\le i\le
n}\Big|\disp\frac{1}{\gamma\pm i+1}\Big|_{v}^{1/n}\Big)\le 1.$$
Combining these estimations with \eqref{down} we get for
$\ell=j,\ldots,k$,
$$\disp\limsup_{n\longrightarrow\infty}
\Big|r_{\gamma\pm n,j-1}^{(k,\ell)}\Big|_{v}^{1/n}\le 1.$$ The case
$\ell=j-1$ is trivial by \eqref{down}. This completes  the proof of
 the lemma.
\end{proof}
In addition,  if we set $r_{\gamma\pm
n,j}^{(k,\ell)}=0$ for $\ell=0,\dots,j-1$, we find, by combining
formulae \eqref{tr1} and \eqref{rho1},
\begin{equation}
\label{tr2}
\begin{aligned}
{\mathcal L}(h_{\gamma+n,k})=&\; \Gamma(\gamma+n+1)\;
x^{-\gamma-n-1}\;\sum_{j=0}^{k}\;\rho^{(k)}_{\gamma+n,j}\;(\log x)^j\\
=&\; (\gamma)_{n+1}\Gamma(\gamma)\; x^{-\gamma-n-1}\Big( (-1)^k(\log
x
)^k\;+\;\sum_{j=0}^{k-1}\;\sum_{\ell=0}^{k}\rho^{(k)}_{\gamma,\ell}\;
r_{\gamma+n,j}^{(k,\ell)}\;(\log x)^j\Big),\\
{\mathcal L}(h_{\gamma-n,k})=&\; \Gamma(\gamma-n+1)\;
x^{-\gamma+n-1}\;\sum_{j=0}^{k}\;\rho^{(k)}_{\gamma-n,j}\;
\;(\log x)^j\\
=&\; \frac{(-1)^n\gamma\Gamma(\gamma)}{(-\gamma)_n}\;
x^{-\gamma+n-1}\Big( (-1)^k(\log x
)^k\;+\;\sum_{j=0}^{k-1}\;\sum_{\ell=0}^{k}\rho^{(k)}_{\gamma,\ell}\;
r_{\gamma-n,j}^{(k,\ell)}\;(\log x)^j\Big).
\end{aligned}
\end{equation}
\begin{remarks}
\label{grow} \rm{(i) According to  the recursion formulae
\eqref{down} and by downward induction on $j$, we find, for all
$0\le j, \ell\le k$ and all $n\in \bN$
\begin{equation}
\label{grow1}
\begin{aligned}
r_{\gamma+n,j}^{(k,\ell)}\in \Big<1,\prod_{1\le i\le
t}\frac{1}{\gamma+m_{i}},\quad 1\le m_i\le n, \quad 1\le t\le
k\Big>_{\bZ}, \\
r_{\gamma-n,j}^{(k,\ell)}\in \Big<1,\prod_{1\le i\le
t}\frac{1}{\gamma-m_{i}},\quad 0\le m_{i}\le n, \quad 1\le t\le
k\Big>_{\bZ}.
\end{aligned}
\end{equation}
(ii)  The formula \eqref{tr1} shows that, for  any $(\gamma,k)\in
K\setminus\bZ\times\bN$,  ${\mathcal L}(h_{\gamma,k})\ne 0$. Hence,
by \eqref{op1}, we have
 ${\mathcal L}(xh_{\gamma,k}) \frac{d}{dx}({\mathcal L}(h_{\gamma,k})) \ne 0
$.\\
(iii) From \eqref{tr2} and by $x$-adic formal completion, the
Laplace transform $\mathcal{L}$ can be extended to injective
$K$-linear maps:}
$$x^\gamma\;K[[x]]\hookrightarrow
x^{-\gamma-1}\; \bC [[\frac{1}{x}]],\quad
x^{\gamma}\;K[[\frac{1}{x}]]\hookrightarrow
x^{-\gamma-1}\;\bC[[x]],$$
$$x^{\gamma}\;K[[x]][\log x
]\hookrightarrow x^{-\gamma-1}\;\bC[[\frac{1}{x}]][\log x],\quad
x^\gamma\;K[[\frac{1}{x}]][\log x]\hookrightarrow
x^{-\gamma-1}\;\bC[[x]][\log x].$$ The injectivity can be proved by
filtering by the degree of the logarithms using \eqref{tr2}.
\end{remarks}
\begin{lemma}
\label{Grow} With the notation of Lemma \ref{rholem}, we have \\
$(1)$ if $\gamma\in K\setminus\bZ$, the absolute value (in the usual
sense)  of $r_{\gamma\;\pm n\;,j}^{(k,\ell)}$ has at most a
geometric growth in $n$ for all $0\le j, \ell\le k$.\\
$(2)$ if $\gamma\in\bQ\setminus\bZ$, the quantities $r_{\gamma+
n,j}^{(k,\ell)}$  and $r_{\gamma- n,j}^{(k,\ell)}$ are rational
numbers for any $n$, and the least common denominator
 of $r_{\gamma,j}^{(k,\ell)},r_{\gamma+1,j}^{(k,\ell)}, \ldots,r_{\gamma+ n,j}^{(k,\ell)}$
$($resp. $r_{\gamma,j}^{(k,\ell)},r_{\gamma-1 ,j}^{(k,\ell)},
\ldots,r_{\gamma- n,j}^{(k,\ell)})$ has at most a geometric growth
in $n$ for all $0\le j, \ell\le k$.
\end{lemma}
\begin{proof} (1) follows by downward induction on $j$ using \eqref{down}
and the fact that $$\log\Big(\Big|\frac{1}{\gamma}\Big|+\Big|
\frac{1}{\gamma+1}\Big|+ . . . +\Big|\frac{1}{\gamma+n}\Big|\Big) =
O(n)\quad \text{and}\quad \log\Big(\Big|\frac{1}{\gamma}\Big|+\Big|
\frac{1}{\gamma-1}\Big|+ . . . +\Big|\frac{1}{\gamma-n+1}\Big|\Big)
= O(n).$$ (2) Let $\gamma=a/b$ with $(a,b)\in \bZ\times\bZ_{>0}$. It
is clear by \eqref{grow1} that the quantities $r_{\gamma+
n,j}^{(k,\ell)}$  and $r_{\gamma- n,j}^{(k,\ell)}$ are rational
numbers for any $n$. In addition, the prime numbers Theorem (see
e.g. \cite{HW}) implies that $\lim_{n\to +\infty}\frac{\log
\text{lcm}\{1,\ldots,n\}}{n} = 1$, where ``lcm'' denotes the least
common multiple. Therefore, there exists a positive constant $C>0$
such that for $n$ sufficiently large, we have
$\text{lcm}\{1,\ldots,n\}\le C^n$. Thus,
$$\text{lcm}(|a|,|a-b|,|a-2b|,\ldots,|a+(1-n)b|)\le
C^{|a|}C^{bn},$$ and
$$\text{lcm}(|a|,|a+b|,|a+2b|,\ldots,|a+nb|)\le C^{|a|}C^{bn}.$$
Hence,
$$\text{lcd}\Big\{\prod_{1\le i\le
t}\frac{1}{\gamma+m_{i}}\quad |\quad 1\le m_i\le n, \quad 1\le t\le
k\Big\}\le C^{k|a|}C^{kbn},$$ and
$$\text{lcd}\Big\{\prod_{1\le i\le
t}\frac{1}{\gamma-m_{i}}\quad |\quad 0\le m_{i}\le n, \quad 1\le
t\le k\Big\}\le C^{k|a|}C^{kbn},$$ where ``lcd'' denotes the least
common denominator. Combining this with \eqref{grow1}, we get the
last statement of (2).
\end{proof}
\subsection{3.2 Standard Laplace transform in several variables}
 In this paragraph, we extend the Laplace transform defined
 in the previous paragraph to
 several variables  while preserving the $K$-linearity, and the
commutation with derivations in the following sense:
\begin{equation}
\label{op2}
\begin{aligned}
\underline{\partial}^{\underline{\alpha}}(\underline{{\mathcal
L}}(h_{\underline{\gamma},\underline{k}}))=&\underline{{\mathcal
L}}((-1)^{|\underline{\alpha}|}\underline{x}^{\underline{\alpha}}h_{\underline{\gamma},\underline{k}}),
\;\;\text{and}\quad \underline{{\mathcal
L}}(\underline{\partial}^{\underline{\alpha}}(h_{\underline{\gamma},\underline{k}}))=
\underline{x}^{\underline{\alpha}}\underline{{\mathcal
 L}}(h_{\underline{\gamma},\underline{k}})\quad \text{for any}\quad \underline{\alpha}\in \bN^d.
\end{aligned}
\end{equation}
With these conditions, the definition of $\underline{{\mathcal
 L}}$ will be interesting just in the case where
 $(\underline{\gamma}, \underline{k})\in
 (K\setminus\bZ)^d\times\bN^d$. Indeed,
  if for some $1\le i_0\le d$, $(\gamma_{i_0},k_{i_0})\in
\bN\times\{0\}$, we will find
\begin{equation*}
\begin{aligned}
x_{i_0}^{\gamma_{i_0}+1}\underline{{\mathcal
L}}(h_{\underline{\gamma},\underline{k}})=\underline{{\mathcal
L}}(\partial_{i_0}^{\gamma_{i_0}+1}
(h_{\underline{\gamma},\underline{k}}))=\underline{{\mathcal
L}}(0)=0.
\end{aligned}
\end{equation*}
For the remainder  of this section, we fix $(\underline{\gamma},
\underline{k})$ in $(K\setminus\bZ)^d\times\bN^d$. For
$$f(\underline{x})=\sum_{\underline{\alpha}\in {\bN}^d}
a_{-\underline{\alpha}}\underline{x}^{-\underline{\alpha}}+\sum_{\underline{\alpha}\in
{\bN}^d} a_{\underline{\alpha}}\underline{x}^{\underline{\alpha}}\in
K[[\underline{x},\frac{1}{\underline{x}}]],\footnote{Note here that the Cauchy multiplication is not involved in this theory.}$$ we define the
Laplace transform $\underline{\mathcal{L}}(f(\underline{x})
h_{\underline{\gamma},\underline{k}})$ of $f(\underline{x})
h_{\underline{\gamma},\underline{k}}$ as follows:
$$\underline{{\mathcal L}}(f(\underline{x})
h_{\underline{\gamma},\underline{k}})=\sum_{\underline{\alpha}\in
{\bN}^d} a_{-\underline{\alpha}}\prod_{1\le i\le d}{\mathcal
L}_i(h_{\gamma_{i}-\alpha_{i},k_{i}})+\sum_{\underline{\alpha}\in
{\bN}^d} a_{\underline{\alpha}}\prod_{1\le i\le d}{\mathcal
L}_i(h_{\gamma_{i}+\alpha_{i},k_{i}}),$$ where ${\mathcal L}_i$
denotes the Laplace transform  defined in the previous paragraph
with respect to the variable $x_i$. It is easy to check, from
\eqref{op1} and (2) of Lemma \ref{Grow}, that
$\underline{\mathcal{L}}$ satisfies the formulae of \eqref{op2}.

Explicitly, if we set
$\Gamma(\underline{\gamma})=\Gamma(\gamma_1)\ldots\Gamma(\gamma_d)$, we get from \eqref{tr2}:
\begin{equation*}
\begin{aligned}
\underline{{\mathcal
L}}(h_{\underline{\gamma}+\underline{\alpha},\underline{k}})
=&(\underline{\gamma})_{\underline{\alpha}+\underline{1}}\Gamma(\underline{\gamma})\;
\underline{x}^{-\underline{\gamma}-\underline{\alpha}-\underline{1}}\prod_{1\le
i\le d}
\Big(\sum_{j_i\;=0}^{k_i}\rho^{(k_i)}_{\gamma_i+\alpha_i,j_i}\;
\;(\log x_i)^{j_i}\Big),\\
=&\;(\underline{\gamma})_{\underline{\alpha}+\underline{1}}\Gamma(\underline{\gamma})\;
\underline{x}^{-\underline{\gamma}-\underline{\alpha}-\underline{1}}\Big(
(-1)^{|\underline{k}|}\;(\log \underline{x})^{\underline{k}}+ \sum_{
\underline{j}<\underline{k}} \Big(\prod_{ 1\le i\le
d}\rho^{(k_i)}_{\gamma_i+\alpha_i,j_i}\Big)\;(\log \underline{x})^{\underline{j}}\Big)\\
=&\;(-1)^{|\underline{k}|}(\underline{\gamma})_{\underline{\alpha}+\underline{1}}\Gamma(\underline{\gamma})\;
h_{-\underline{\gamma}-\underline{\alpha}-\underline{1},\underline{k}}\\
&\; +\;
\;(\underline{\gamma})_{\underline{\alpha}+\underline{1}}\Gamma(\underline{\gamma})\;
\underline{x}^{-\underline{\gamma}-\underline{\alpha}-\underline{1}}\sum_{
\underline{j}<\underline{k}} \Big(\sum_{\underline{\ell}\;\le
\underline{k}}\prod_{ 1\le i\le d}\rho^{(k_i)}_{\gamma_i,\ell_i}\;
r_{\gamma_i+\alpha_i,j_i}^{(k_i,\ell_i)}\Big)\;(\log \underline{x})^{\underline{j}}\\
\underline{{\mathcal
L}}(h_{\underline{\gamma}-\underline{\alpha},\underline{k}})=&
\frac{(-1)^{|\underline{\alpha}|}\Gamma(\underline{\gamma})(\underline{\gamma})_{\underline{0}}}
{(-\underline{\gamma})_{\underline{\alpha}}}\;
\underline{x}^{-\underline{\gamma}+\underline{\alpha}-\underline{1}}\prod_{1\le
i\le d}
\Big(\sum_{j_i\;=0}^{k_i}\rho^{(k_i)}_{\gamma_i-\alpha_i,j_i}\;\;(\log x_i)^{j_i}\Big)\\
=&\;\frac{(-1)^{|\underline{\alpha}|}\Gamma(\underline{\gamma})(\underline{\gamma})_{\underline{0}}}
{(-\underline{\gamma})_{\underline{\alpha}}}\;
\underline{x}^{-\underline{\gamma}+\underline{\alpha}-\underline{1}}\Big(
(-1)^{|\underline{k}|}\;(\log \underline{x})^{\underline{k}}+ \sum_{
\underline{j}<\underline{k}} \Big(\prod_{ 1\le i\le
d}\rho^{(k_i)}_{\gamma_i-\alpha_i,j_i}\Big)\;(\log
\underline{x})^{\underline{j}}\Big)
\\
=&\;
\;\frac{(-1)^{|\underline{\alpha}|+|\underline{k}|}\Gamma(\underline{\gamma})(\underline{\gamma})_{\underline{0}}}
{(-\underline{\gamma})_{\underline{\alpha}}}\;h_{-\underline{\gamma}+\underline{\alpha}-\underline{1},\underline{k}}\\
&\;+\;\frac{(-1)^{|\underline{\alpha}|}\Gamma(\underline{\gamma})(\underline{\gamma})_{\underline{0}}}
{(-\underline{\gamma})_{\underline{\alpha}}}\;\;
\underline{x}^{-\underline{\gamma}+\underline{\alpha}-\underline{1}}\sum_{
\underline{j}<\underline{k}} \Big(\sum_{\underline{\ell}\;\le
\underline{k}}\prod_{ 1\le i\le d}\rho^{(k_i)}_{\gamma_i,\ell_i}\;
r_{\gamma_i-\alpha_i,j_i}^{(k_i,\ell_i)}\Big)\;(\log
\underline{x})^{\underline{j}}.
\end{aligned}
\end{equation*}
If we set
$\rho^{(\underline{k})}_{\underline{\gamma},\underline{\ell}}=\prod_{
1\le i\le d}\rho^{(k_i)}_{\gamma_i,\ell_i}$,
$r_{\underline{\gamma}-\underline{\alpha},\underline{j}}^{(\underline{k},\underline{\ell})}=\prod_{
1\le i\le d} r_{\gamma_i-\alpha_i,j_i}^{(k_i,\ell_i)}$ and
$r_{\underline{\gamma}+\underline{\alpha},\underline{j}}^{(\underline{k},\underline{\ell})}=\prod_{
1\le i\le d} r_{\gamma_i+\alpha_i,j_i}^{(k_i,\ell_i)}$, the
equalities above become
\begin{equation}\label{tr3}
\begin{aligned}
\underline{{\mathcal
L}}(h_{\underline{\gamma}+\underline{\alpha},\underline{k}})=&
\;(-1)^{|\underline{k}|}(\underline{\gamma})_{\underline{\alpha}+\underline{1}}\Gamma(\underline{\gamma})\;
h_{-\underline{\gamma}-\underline{\alpha}-\underline{1},\underline{k}}\\
&\; +\;
\;(\underline{\gamma})_{\underline{\alpha}+\underline{1}}\Gamma(\underline{\gamma})\;
\underline{x}^{-\underline{\gamma}-\underline{\alpha}-\underline{1}}\sum_{
\underline{j}<\underline{k}} \Big(\sum_{\underline{\ell}\;\le
\underline{k}}\rho^{(\underline{k})}_{\underline{\gamma},\underline{\ell}}\;
r_{\underline{\gamma}+\underline{\alpha},\underline{j}}^{(\underline{k},\underline{\ell})}\Big)
\;(\log \underline{x})^{\underline{j}}\\
\underline{{\mathcal
L}}(h_{\underline{\gamma}-\underline{\alpha},\underline{k}})=& \;
\;\frac{(-1)^{|\underline{\alpha}|+|\underline{k}|}\Gamma(\underline{\gamma})(\underline{\gamma})_{\underline{0}}}
{(-\underline{\gamma})_{\underline{\alpha}}}\;h_{-\underline{\gamma}+\underline{\alpha}-\underline{1},\underline{k}}\\
&\;+\;\frac{(-1)^{|\underline{\alpha}|}\Gamma(\underline{\gamma})(\underline{\gamma})_{\underline{0}}}
{(-\underline{\gamma})_{\underline{\alpha}}}\;\;
\underline{x}^{-\underline{\gamma}+\underline{\alpha}-\underline{1}}\sum_{
\underline{j}<\underline{k}} \Big(\sum_{\underline{\ell}\;\le
\underline{k}}\rho^{(\underline{k})}_{\underline{\gamma},\underline{\ell}}\;
r_{\underline{\gamma}-\underline{\alpha},\underline{j}}^{(\underline{k},\underline{\ell})}\Big)\;(\log
\underline{x})^{\underline{j}}.
\end{aligned}
\end{equation}
\begin{remark}
\label{inj} \rm{From \eqref{tr3} and by  $(x_1,\ldots,x_d)$-adic
formal completion, the Laplace transform $\underline{\mathcal{L}}$
extends to injective $K$-linear maps:}
$$\underline{x}^{\underline{\gamma}}\;K[[\underline{x}]]\hookrightarrow
\underline{x}^{-\underline{\gamma}-\underline{1}}\; \bC
[[\frac{1}{\underline{x}}]],\quad
\underline{x}^{\underline{\gamma}}\;K[[\frac{1}{\underline{x}}]]\hookrightarrow
\underline{x}^{-\underline{\gamma}-\underline{1}}\;\bC[[\underline{x}]],$$
$$\underline{x}^{\underline{\gamma}}\;K[[\underline{x}]][\log x_1,\ldots,\log x_d
]\hookrightarrow
\underline{x}^{-\underline{\gamma}-\underline{1}}\;\bC[[\frac{1}{\underline{x}}]][\log x_1,\ldots,\log x_d],$$
$$ \underline{x}^{\underline{\gamma}}\;K[[\frac{1}{\underline{x}}]][\log
x_1,\ldots,\log x_d]\hookrightarrow
\underline{x}^{-\underline{\gamma}-\underline{1}}\;\bC[[\underline{x}]][\log
x_1,\ldots,\log x_d].$$ The injectivity follows by filtering by the
 order (defined in section 2) of the multi-exponent   of
$\log\underline{x}$ using \eqref{tr3}.
\end{remark}
\subsection{3.3 Arithmetic estimates} In this paragraph, for $v\in \Sigma_{\rm f}$ and $f\in
K[[\underline{x}]]$, we give, in term of $r_v(f)$, lower bounds of
$p(v)$-adic radius of convergence of the formal power series over
$K$ occurring  in
$\underline{\mathcal{L}}(fh_{\underline{\gamma},\underline{k}})$ and
$\underline{\mathcal{L}}\Big(f\Big(\frac{1}{\underline{x}}\Big)h_{\underline{\gamma},\underline{k}}\Big)$.
\begin{lemma}
\label{Re}With the notation of the previous paragraph $($in
particular formula
\eqref{tr3}$)$, we have \\
 1) For any $v\in \Sigma_{\rm{f}}$, $$\limsup_{|\underline{\alpha}|\to
+\infty}\Big|r_{\underline{\gamma}+\underline{\alpha},\underline{j}}^{(\underline{k},\underline{\ell})}
\Big|_v^{1/|\underline{\alpha}|}\le 1\quad\text{ and}\quad
\limsup_{|\underline{\alpha}|\to
+\infty}\Big|r_{\underline{\gamma}-\underline{\alpha},\underline{j}}^{(\underline{k},\underline{\ell})}
\Big|_v^{1/|\underline{\alpha}|}\le 1.$$
 2) The absolute value $($in the
usual sense$)$ of
 $r_{\underline{\gamma}+\underline{\alpha},\underline{j}}^{(\underline{k},\underline{\ell})}$
 has at most a geometric growth in $|\underline{\alpha}|$.\\
3) If $\underline{\gamma}\in(\bQ\setminus\bZ)^d$, the quantities
$r_{\underline{\gamma}\;\pm
\underline{\alpha},\underline{j}}^{(\underline{k},\underline{\ell})}$
are rational numbers for all $\underline{\alpha}\in\bN^d$, and the
least common denominator of
$\Big\{r_{\underline{\gamma}+\underline{\alpha},\underline{j}}^{(\underline{k},\underline{\ell})}\;
\Big|\; |\underline{\alpha}|\le n\Big\}$ $\Big($resp.
$\Big\{r_{\underline{\gamma}-\underline{\alpha},\underline{j}}^{(\underline{k},\underline{\ell})}\;
\Big|\; |\underline{\alpha}|\le n\Big\}\Big)$ has at most a
geometric growth in $n$.
\end{lemma}
\begin{proof} 1) By Lemma \ref{rholem}, we have, for $i=1,\ldots,d$,
$ \limsup_{\alpha_i\to
+\infty}\Big|r_{\gamma_i\pm\alpha_i,j_i}^{(k_i,\ell_i)}\Big|_v^{1/|\alpha_i|}\le
1$, and hence $\limsup_{|\underline{\alpha}|\to
+\infty}\Big|r_{\gamma_i\pm\alpha_i,j_i}^{(k_i,\ell_i)}
\Big|_v^{1/|\underline{\alpha}|}\le 1$. The product  over all the
$1\le i\le d$ of these quantities gives the desired inequalities.
The rest of the lemma results from Lemma \ref{Grow} and the
definitions above of the quantities
$r_{\underline{\gamma}\pm\underline{\alpha},\underline{j}}^{(\underline{k},\underline{\ell})}$.
\end{proof}
 The arithmetic properties of
$\underline{\mathcal{L}}$ are based on the following Lemma which
generalizes a well known identity in the $p$-adic analysis.
\begin{lemma}\label{liouv}
Let $v\in \Sigma_{\rm f}$. Assume that $\gamma_1,\ldots,\gamma_d$
are $($non-Liouville$)$ numbers of $K\cap {\bZ}_{p(v)}\setminus
{\bZ}$. Then $$\lim_{|\underline{\alpha}|\to
+\infty}\Big|(\underline{\gamma})_{\underline{\alpha}+\underline{1}}\Big|_v^{1/|\underline{\alpha}|}
=\pi_{v}.$$ In particular, for $d=1$, we have
 $\disp\lim_{\alpha_1\to
+\infty}|(\gamma_1)_{\alpha_1}|_v^{1/\alpha_1} =\pi_{v}$.
\end{lemma}
\begin{proof}
Combining the inequalities (12) and (14) of \cite{Cl}, we find that,
for any $i=1,\ldots,d$, there exist two real numbers $e_i,e'_i$ such
that
$$|p(v)|_v^{(\alpha_i/(p(v)-1)+e_i\log(1+\alpha_i)+e'_i)}\le
 |\gamma_i^{-1} (\gamma_i)_{(\alpha_i+1)}|_v\le
\frac{|p(v)|_v^{\alpha_i/(p(v)-1)}}{(\alpha_i+1)},\;\; \text{for
any}\;\;\alpha_i\in {\bN}.
$$
Thus,
$$|p(v)|_v^{(|\underline{\alpha}|/(p(v)-1)+\sum_{1\le i\le d}
(e_i\log(1+\alpha_i)+e'_i))}\Big|\prod_{1\le i\le
d}\gamma_i\Big|_v\le
 |(\underline{\gamma})_{\underline{\alpha}+\underline{1}}|_v\le
\prod_{1\le i\le
d}\frac{|\gamma_i|_v}{(\alpha_i+1)}|p(v)|_v^{|\underline{\alpha}|/(p(v)-1)},$$
for any $\underline{\alpha}\in {\bN}^d$. Hence, the lemma results
from the fact:
$$\Big|\sum_{1\le i\le d} e_i\log(1+\alpha_i)\Big|\le
 d\max_{1\le i\le d}(1,|e_i|)\log(1+|\underline{\alpha}|).$$
\end{proof}
Before stating the following proposition, let us recall that the notations $\mathcal R_v^s(f)$, $r_v(f)$ and $\mathcal F$ are defined in sections 1 and 2.
\begin{proposition}\label{pro1} Let $v$ be a finite place of $\Sigma_{\rm f}$
and let  $f=\sum_{\underline{\alpha}\in {\bN}^d}
a_{\underline{\alpha}}\underline{x}^{\underline{\alpha}}\in
K[[\underline{x}]]$ be a power series.
 Assume  $\underline{\gamma}\in
(K\cap\;\bZ_{p(v)}\setminus\bZ)^d$. Then there exist power series
$f_{\underline{\gamma},\underline{k},\underline{j}}\in
{\bC}\otimes_{K}{\mathcal R}_v^{-1}(f)\;\;($resp.
$f_{\underline{\gamma},\underline{k},\underline{j}}^*\in
{\bC}\otimes_{K}{\mathcal R}_v^{1}(f)),\; \underline{j}\le
\underline{k}$, which satisfy the following conditions
\begin{equation}
\label{eqpro1}
\begin{aligned}
 \underline{{\mathcal
L}}\Big(fh_{\underline{\gamma},\underline{k}}\Big)=&\;
\underline{x}^{-\underline{\gamma}-\underline{1}}\underline{\Gamma}(\underline{\gamma})
\disp\sum_{\underline{j}\le
\underline{k}}f_{\underline{\gamma},\underline{k},\underline{j}}\Big(\frac{1}{\underline{x}}\Big)(\log
\underline{x})^{\underline{j}}\\
\underline{{\mathcal
L}}\Big(f\Big(\frac{1}{\underline{x}}\Big)h_{\underline{\gamma},\underline{k}}\Big)=&\;
\underline{x}^{-\underline{\gamma}-\underline{1}}\underline{\Gamma}(\underline{\gamma})
\disp\sum_{\underline{j}\le
\underline{k}}f_{\underline{\gamma},\underline{k},\underline{j}}^*\;(\log
\underline{x})^{\underline{j}}
\end{aligned}
\end{equation}
with
$f_{\underline{\gamma},\underline{k},\underline{k}}f_{\underline{\gamma},\underline{k},\underline{k}}^*\ne
0$ such that
$r_{v}(f_{\underline{\gamma},\underline{k},\underline{k}})=r_{v}(f)\pi_{v}^{-1}\;\;($resp.
$r_{v}(f_{\underline{\gamma},\underline{k},\underline{k}}^*)=r_{v}(f)\pi_{v})$.
Moreover, if $fh_{\underline{\gamma},\underline{k}}\;\;($resp.
$f(\frac{1}{\underline{x}})h_{\underline{\gamma},\underline{k}})$ is
solution of a differential equation $\phi$ of
$K[x_1,\ldots,x_d,\partial_1,\ldots,\partial_d]$,  then
$\mathcal{L}(fh_{\underline{\gamma},\underline{k}})\;\;($resp.
$\mathcal{L}(f(\frac{1}{\underline{x}})h_{\underline{\gamma},\underline{k}}))$
is solution of $\mathcal{F}(\phi)$, where $\mathcal{F}$ is the Fourier-Laplace transform.
\end{proposition}
\begin{proof} Set, for all $|\underline{j}|\le |\underline{k}|$,
\begin{equation}
\label{klj}
\begin{aligned}
f_{\underline{\gamma},\underline{k},\underline{k}}=&\;
\sum_{\underline{\alpha}\in
\bN^d}(-1)^{|\underline{k}|}\;a_{\underline{\alpha}}\;(\underline{\gamma})_{\underline{\alpha}+\underline{1}}
\;\underline{x}^{\underline{\alpha}}\\
f_{\underline{\gamma},\underline{k},\underline{j}}=&\;\sum_{\underline{\ell}\le
\underline{k}}\rho^{(\underline{k})}_{\underline{\gamma},\underline{\ell}}\;\sum_{\underline{\alpha}\in
\bN^d}(-1)^{|\underline{k}|}\;a_{\underline{\alpha}}\;(\underline{\gamma})_{\underline{\alpha}+\underline{1}}
\;r_{\underline{\gamma}+\underline{\alpha},\underline{j}}^{(\underline{k},\underline{\ell})}
\;\underline{x}^{\underline{\alpha}} \\
f_{\underline{\gamma},\underline{k},\underline{k}}^*=&\;\sum_{\underline{\alpha}\in
\bN^d}\frac{(-1)^{|\underline{\alpha}|+|\underline{k}|}\;(\underline{\gamma})_{\underline{0}}}
{(-\underline{\gamma})_{\underline{\alpha}}}\;a_{\underline{\alpha}}\;\underline{x}^{\underline{\alpha}} \\
f_{\underline{\gamma},\underline{k},\underline{j}}^*=&\;\sum_{\underline{\ell}\le
\underline{k}}\rho^{(\underline{k})}_{\underline{\gamma},\underline{\ell}}\;
\sum_{\underline{\alpha}\in \bN^d}
\frac{(-1)^{|\underline{\alpha}|+|\underline{k}|}\;(\underline{\gamma})_{\underline{0}}}
{(-\underline{\gamma})_{\underline{\alpha}}}\;a_{\underline{\alpha}}\;
r_{\underline{\gamma}-\underline{\alpha},\underline{j}}^{(\underline{k},\underline{\ell})}
\;\underline{x}^{\underline{\alpha}}.
\end{aligned}
\end{equation}
By \eqref{tr3}, and in virtue of Lemmas \ref{Re} and \ref{liouv},
one sees that these power series satisfy the desired conditions. The
last statement results from \eqref{op2}.
\end{proof}
\subsection{3.4 Laplace transform and arithmetic Gevrey series}
 In this paragraph, we shall explain the action of the standard Laplace transform in several
 variables on the arithmetic Gevrey series. Let us first give the definition of
 these power series:

Fix an embedding $\overline{\bQ}\hookrightarrow \bC$, and let
$(a_{\underline{\alpha}})_{\underline{\alpha}\in \bN^d}$ be a family  of
algebraic numbers of $\overline{\bQ}$. Consider the following
conditions:
\begin{itemize}
 \item[($A_1$):]   for
all $\underline{\alpha}=(\alpha_1,\ldots,\alpha_d)\in {\bN}^d$,
there exists a constant $C_1\in {\bR}_{>0}$ such that
 $a_{\underline{\alpha}}$ and its conjugates
over ${\bQ}$ do not exceed $C_1^{|\underline{\alpha}|}$ in absolute
value;
\item[($A_2$):] there exists a constant $C_2\in {\bR}_{>0}$ such that the
common denominator in ${\bN}$ of $\{a_{\underline{\alpha}},
|\underline{\alpha}|\le n\}$ does not exceed $C_2^{n+1}$.
\end{itemize}
\textbf{Definition.} An \emph{arithmetic Gevrey series
 of order} $s\in {\bQ}$ is an element  $f=\sum_{\underline{\alpha}\in{\bN}^d}
 a_{\underline{\alpha}}\underline{x}^{\underline{\alpha}}$ of
 $K[[\underline{x}]]$ such that
 the sequence $(a_{\underline{\alpha}}/(\underline{\alpha}!)^{s})_{\underline{\alpha}\in \bN^d}$
satisfies the conditions $(A_1)$ and $(A_2)$.

The power series $f=\sum_{\underline{\alpha}\in{\bN}^d}
 a_{\underline{\alpha}}\underline{x}^{\underline{\alpha}}$ is called
 a $G$\emph{-function} (resp. $E$-\emph{function}) if it is an arithmetic Gevrey series
 of order $0$ (resp. $-1$) and satisfies the following holonomicity condition:
\begin{itemize}
 \item[($H$):] $f$ is called rationally holonomic  over $K(\underline{x})$, if the
$K(\underline{x})/K$-differential module generated by
$(\underline{\partial}^{\underline{\alpha}}(f))_{\underline{\alpha}\in{\bN}^d}$ is a
$K(\underline{x})$-vector space of finite dimension.
\end{itemize}
The condition ($H$)  is equivalent to say that $f$ is a solution of a linear partial differential equation with coefficients in $K(\underline{x})$.\\ \ \\
\textbf{Notation.} Let $s$ be a  rational number. We denote
$K\{\underline{x}\}_{s}$ the set of the power series
$f=\sum_{\underline{\alpha}\in{\bN}^d}a_{\underline{\alpha}}\underline{x}^{\underline{\alpha}}\in
K[[\underline{x}]]$ such that  the sequence
$(a_{\underline{\alpha}}/(\underline{\alpha}!)^{s})_{\underline{\alpha}\in \bN^d}$
 satisfies $(A_1)$ and ($A_2$).
$$K\Big\{\frac{1}{\underline{x}}\Big\}_{s}=\Big\{\sum_{\underline{\alpha}\in{\bN}^d}
 a_{\underline{\alpha}}\underline{x}^{-\underline{\alpha}}\;|\; \sum_{\underline{\alpha}\in{\bN}^d}
 a_{\underline{\alpha}}\underline{x}^{\underline{\alpha}}\in K\{\underline{x}\}_{s}\Big\},
 $$
$$NGA\{\underline{x}\}_{s}^K= \Big\{\sum_{\text{ finite
sum}}f_{{\underline{\gamma}},\underline{k}}h_{{\underline{\gamma}},\underline{k}}\;\Big|\;
f_{{\underline{\gamma}},\underline{k}}\in
K\{\underline{x}\}_{s},\;\;(\underline{\gamma},\underline{k})\in
(\bQ\setminus\bZ)^d\times\bN^d\Big\},$$
$$NGA\Big\{\frac{1}{\underline{x}}\Big\}_{s}^K=
\Big\{\sum_{\text{ finite
sum}}f_{{\underline{\gamma}},\underline{k}}h_{{\underline{\gamma}},\underline{k}}\;\Big|\;
f_{{\underline{\gamma}},\underline{k}}\in
K\Big\{\frac{1}{\underline{x}}\Big\}_{s},\;\;(\underline{\gamma},\underline{k})\in
(\bQ\setminus\bZ)^d\times\bN^d\Big\}.$$  Following \cite[section
1]{An}, we see that these two latter sets are differential
$K[\underline{x}]$-algebras. The elements of
$NGA\{\underline{x}\}_{s}^K$ will be called \emph{Nilson-Gevrey
power series of order $s$}. The following result expresses how the
Laplace transform $\mathcal{L}$ acts on these differential
$K[\underline{x}]$-algebras.
\begin{proposition}\label{gev}
The Laplace transform $\underline{\mathcal{L}}$ induces the
following injective $K$-linear maps
$$NGA\{\underline{x}\}_{s}^K \hookrightarrow \bC\otimes_K NGA\Big\{\frac{1}{\underline{x}}\Big\}_{s+1}^K,$$
$$NGA\Big\{\frac{1}{\underline{x}}\Big\}_{s}^K \hookrightarrow \bC\otimes_K NGA\{\underline{x}\}_{s-1}^K.$$
\end{proposition}
\begin{proof} Let  $f=\sum_{\underline{\alpha}\in {\bN}^d}
a_{\underline{\alpha}}\underline{x}^{\underline{\alpha}}\in
K\{\underline{x}\}_{s}$ and $(\underline{\gamma},\underline{k})\in
(\bQ\setminus\bZ)^d\times\bN^d$. By Proposition \ref{pro1}, we have
\begin{equation}
\label{eqpro1}
\begin{aligned}
 \underline{{\mathcal
L}}\Big(fh_{\underline{\gamma},\underline{k}}\Big)=&\;
\underline{x}^{-\underline{\gamma}-\underline{1}}\underline{\Gamma}(\underline{\gamma})
\disp\sum_{\underline{j}\le
\underline{k}}f_{\underline{\gamma},\underline{k},\underline{j}}\Big(\frac{1}{\underline{x}}\Big)(\log
\underline{x})^{\underline{j}}\\
\underline{{\mathcal
L}}\Big(f\Big(\frac{1}{\underline{x}}\Big)h_{\underline{\gamma},\underline{k}}\Big)=&\;
\underline{x}^{-\underline{\gamma}-\underline{1}}\underline{\Gamma}(\underline{\gamma})
\disp\sum_{\underline{j}\le
\underline{k}}f_{\underline{\gamma},\underline{k},\underline{j}}^*\;(\log
\underline{x})^{\underline{j}}
\end{aligned}
\end{equation}
where the power series
$f_{\underline{\gamma},\underline{k},\underline{j}}$ and
$f_{\underline{\gamma},\underline{k},\underline{j}}^*$ ($
\underline{j}\le \underline{k}$) are defined by \eqref{klj}. Since
the quantities $(\underline{\gamma})_{\underline{\alpha}}$ and
$r_{\underline{\gamma}\;\pm\;\underline{\alpha},\underline{j}}^{(\underline{k},\underline{\ell})}$
are rational numbers for all $\underline{\alpha}\in \bN^d$, and all
$\underline{j}, \underline{\ell}\le \underline{k}$, then we deduce
from  2) and 3) of Lemma \ref{Re} and Lemma \ref{Si} bellow, that
all the power series
$f_{\underline{\gamma},\underline{k},\underline{j}}$ (resp.
$f_{\underline{\gamma},\underline{k},\underline{j}}^*$) are in
$\bC\otimes_K NGA\Big\{\frac{1}{\underline{x}}\Big\}_{s+1}^K$ (resp.
$\bC\otimes_K NGA\{\underline{x}\}_{s-1}^K$). This implies, by the
$K$-linearity of $\underline{\mathcal{L}}$,
$\underline{\mathcal{L}}(NGA\{\underline{x}\}_{s}^K) \subseteq
\bC\otimes_K NGA\Big\{\frac{1}{\underline{x}}\Big\}_{s+1}^K$ and
$\underline{\mathcal{L}}\Big(NGA\Big\{\frac{1}{\underline{x}}\Big\}_{s}^K\Big)
\subseteq \bC\otimes_K NGA\{\underline{x}\}_{s-1}^K$. The
injectivity follows from Remark \ref{inj}.
\end{proof}
\begin{lemma}\cite{Si}
\label{Si} Let $a_1,\ldots,a_d,b_1,\ldots,b_{d'}$ be rational
numbers in $\bQ\setminus\bZ_{\le 0}$. Then there exists a positive
constant $C>0$ such that for any positive integer $n$,
$${\rm lcd}\Big(\frac{\prod_{i}(a_i)_0}{\prod_{j}(b_j)_0},\ldots,\frac{\prod_{i}(a_i)_n}{\prod_{j}(b_j)_n}\Big)
< n!^{d'-d}C^n,$$ where `` {\rm lcd}'' denotes the least common
denominator.
\end{lemma}
\textbf{Question 3.10.} Does the Laplace transform preserve the
holonomicity condition $(H)$ above?
\section{Formal Laplace transform}
\subsection{4.1. Review of the one-variable case}
In this paragraph, we recall the formal Laplace transform in one
variable as it was defined in \cite[\S5]{MR}. This formal transformation
allows to avoid the transcendental coefficients as those arising in
the Laplace transform seen in the previous section. Moreover, this
transformation has properties of commuting to derivation and
therefore sends a basis of logarithmic solutions at $0$ (resp. at
infinity) of a differential equation $\phi\in K[x,d/dx]$ to
logarithmic solutions at infinity (resp. at $0$) of
$\mathcal{F}_{\tau}(\phi)$ (for any $\tau\in K\setminus\{0\}$).

Let $\nu$ be a positive integer, $\tau$ an element of
$K\setminus\{0\}$ and $\Lambda$ an $\nu\times \nu$ matrix with
entries in $K$ such that all its eigenvalues belong to $K\setminus
\bZ$. Then the matrix
$$x^\Lambda=\exp(\Lambda\log x)=\sum_{n\ge 0}\frac{\Lambda^n(\log x)^n}{n!}\in{\GL}_\nu(K[[\log x]])$$
which satisfies
\begin{equation*}
\frac{d}{dx}(x^\Lambda) = \Lambda x^{-1}x^{\Lambda} = \Lambda
x^{\Lambda-{\bI}_\nu}.
\end{equation*}
For any integer $\mu\ge 1$ and any  $\mu\times \nu$ matrix $Y(x) =
\sum_{n=-\infty}^\infty Y_i x^n$ with entries in $K[[x,1/x]]$, the
\emph{Laplace  transform  ${\mathcal
L}^{\tau}_\Lambda(Y(x)x^\Lambda)$ of $f:=Y(x)x^\Lambda$}, with
respect to $\Lambda$  and to ${\tau}$,  is defined by
\begin{equation}
\label{trf} {\mathcal L}^{\tau}_\Lambda(Y(x)x^\Lambda) = {\mathcal
L}^{\tau}_\Lambda\Big(\sum_{n=-\infty}^\infty Y_n
x^{\Lambda+n{\bI}_\nu}\Big) := \sum_{n=-\infty}^\infty Y_n
C_{\Lambda,\tau}(n) x^{-\Lambda-(n+1){\bI}_\nu}
\end{equation}
where $C_{\Lambda,\tau}\colon\bZ\to\GL_n(K)$ is defined by the
following conditions
\begin{equation*}
C_{\Lambda,\tau}(n) =
  \begin{cases}
  \tau^{-n} (\Lambda+n{\bI}_\nu)(\Lambda+(n-1){\bI}_\nu)\cdots (\Lambda+{\bI}_\nu) &\mbox{si
$n\ge 1$,}\\
   {\bI}_\nu &\mbox{si $n=0$,}\\
   \tau^{-n} \Lambda^{-1}(\Lambda-{\bI}_\nu)^{-1}\cdots(\Lambda+(n+1){\bI}_\nu)^{-1}
&\mbox{si
   $n\le -1$.}
  \end{cases}
\end{equation*}
Notice that  $C_{\Lambda,\tau}$, with these properties, satisfies
\begin{equation}\label{invC}
 C_{\Lambda,\tau}(n) C_{-\Lambda,\tau}(-n-1) =
(-1)^{n+1}\tau\Lambda^{-1}.
\end{equation}
The transformation ${\mathcal L}^{\tau}_\Lambda$ has the following
formal properties \cite[5.1.1]{MR}:
\begin{equation}
\label{Ld=xL} {\mathcal L}^{\tau}_{-\Lambda}\big({\mathcal
L}^{\tau}_\Lambda(f)\big) = - \tau Y(-x)\Lambda^{-1} x^\Lambda,
\quad {\mathcal L}^{\tau}_\Lambda\Big(\frac{df}{dx}\Big) = {\tau}x
{\mathcal L}^{\tau}_\Lambda(f)\;\; \text{and}\;\; {\mathcal
L}^{\tau}_\Lambda(xf) = -\frac{1}{{\tau}}\frac{d}{dx} {\mathcal
L}^{\tau}_\Lambda(f).
\end{equation}
Moreover, if $\mu=1$ and if the entries of $f$  are solutions of a
differential equation $\phi\in K[x,d/dx]$, then those of ${\mathcal
L}^{\tau}_\Lambda(f)$ are solutions of ${\mathcal F}_{\tau}(\phi)$.
\subsection{4.2 Arithmetic estimates}
For the remainder of this section,  $v$ denotes a fixed finite place
of $\Sigma_{\rm f}$. For any matrix $M$ with entries in $K$, we
denote by $\|M\|_v$ the maximum of the $v$-adic absolute values of
the entries of $M$. Let $\Lambda\in\GL_\nu(\bQ)$ be an invertible
matrix such that all its eigenvalues lie in
$\bQ\cap(\bZ_{p(v)}\setminus\bZ)$. In this paragraph, we give upper
and lower bounds of $C_{\Lambda,\tau}(n)$ with respect to the norm
$\|.\|_v$ for any $n\in \bZ$.
\begin{lemma}\label{lem1}
Let $\Lambda\in\GL_\nu(\bQ)$ be an invertible matrix such that all
its  eigenvalues $\gamma_1,\ldots,\gamma_s$ lie in
$\bQ\cap(\bZ_{p(v)}\setminus\bZ)$ $(s\le \nu)$.  Then, there exist
two positive real numbers $c_1,c_2$ such that for any $n\ge 1$
\begin{equation}
\label{eqlem1}
\begin{aligned}
  c_1|\tau|_v^{-n}\max_{1\le j\le s}\{|(\gamma_j+1)_n|_v\} \le \|C_{\Lambda,\tau}(n)\|_v
  \le c_2 n^{\nu-1}|\tau|_v^{-n}\max_{1\le j\le
  s}\{|(\gamma_j+1)_n|_v\}.
\end{aligned}
\end{equation}
In particular,
$$\lim_{n\to+\infty}\|C_{\Lambda,\tau}(n)\|_v^{1/n}=|\tau|_v^{-1}\pi_{v}.$$
\end{lemma}
\begin{proof}
Since  the eigenvalues of $\Lambda$ are all rational numbers, there
exists $U\in\GL_n(\bQ)$  such that the  product
$\Delta=U^{-1}\Lambda U$  is in Jordan form. Setting
$C_{\Delta,\tau}=U^{-1}C_{\Lambda,\tau}U$, there exist two positive
real numbers $c_0,c_1$, such that $c_1\le
\|C_{\Lambda,\tau}(n)\|_v/\|C_{\Delta,\tau}(n)\|_v\le c_0$ for all
$n\in \bZ$. In addition, the matrix $\Delta$ is  a  block diagonal
matrix with blocks
$J_1=\gamma_1{\bI}_{\nu_1}+N_1,\dots,J_s=\gamma_s{\bI}_{\nu_s}+N_s$
on the diagonal (with $\nu_1+\ldots+\nu_s=\nu$ and $N_1,\ldots,N_s$
are nilpotent matrices). Hence, for any $n\in \bZ$,
$C_{\Delta,\tau}(n)$ is a block diagonal matrix with blocks
$C_{J_1,\tau}(n), \dots, C_{J_s,\tau}(n)$ on the diagonal and
$\|C_{\Delta,\tau}(n)\|_v = \max_{1\le j\le s}
\|C_{J_j,\tau}(n)\|_v$. One the other hand, for $j=1,\ldots,s$, we
have $(N_j)^\nu=0$ and  for any $n\ge 1$,
\begin{equation}
  \label{dec:Cgamma}
\begin{aligned}
C_{J_j,\tau}(n)
&= \tau^{-n}\prod_{\ell=1}^n ((\gamma_j+\ell){\bI}_{\nu_j}+N_j) \\
&= \tau^{-n}(\gamma_j+1)_n
    \Big( {\bI}_{\nu_j} + \sum_{t=1}^{\nu_j-1}
    \Big( \sum_{1\le \ell_1<\cdots<\ell_t\le n}
       \frac{1}{(\gamma_j+\ell_1)\cdots(\gamma_j+\ell_t)}
    \Big) N_j^t \Big).
\end{aligned}
\end{equation}
Now, for any integer $\ell\ge 1$,  the sum $\gamma_j+\ell$ is a
rational number for which the  absolute value  of numerator (in the
usual sense)  is bounded above  by $\theta_j\ell$, for a constant
$\theta_j>0$ which only depends  on $\gamma_j$, and for which the
denominator is prime
 to $p(v)$. We deduce $|\gamma_j+\ell|_v \ge (\theta_j\ell)^{-1}$ for any $\ell\ge 1$
and hence, the decomposition \eqref{dec:Cgamma} implies, for any
$n\ge 1$,
\begin{equation}
  \label{CJj}
\begin{aligned}
  |\tau^{-n}(\gamma_j+1)_n|_v \le \|C_{J_j,\tau}(n)\|_v \le (\theta_jn)^{\nu_j-1}
  |\tau^{-n}(\gamma_j+1)_n|_v.
\end{aligned}
\end{equation}
The left  inequality results from the fact that all the elements of
the  diagonal of $C_{J_j,\tau}(n)$  are equals to
$\tau^{-n}(\gamma_j+1)_n$. This last observation gives $\det
C_{J_j,\tau}(n) = \tau^{-n\nu_j}\Big((\gamma_j+1)_n\Big)^{\nu_j}$
and $\det C_{\Lambda,\tau}(n) =\det C_{\Delta,\tau}(n)=
\tau^{-n\nu}\prod_{1\le j\le s} \Big((\gamma_j+1)_n\Big)^{\nu_j}$.
In addition, the inequalities of \eqref{CJj} imply,
\begin{equation*}
  %\label{dec:CDelta}
\begin{aligned}
  |\tau|_v^{-n}\max_{1\le j\le s}\{|(\gamma_j+1)_n|_v \}\le \|C_{\Delta,\tau}(n)\|_v
  \le n^{\nu-1}|\tau|_v^{-n}\max_{1\le j\le s}\{\theta_j^{\nu_j-1}|(\gamma_j+1)_n|_v\},
\end{aligned}
\end{equation*}
and therefore,
\begin{equation}
  %\label{dec:CDelta}
\begin{aligned}
  c_1|\tau|_v^{-n}\max_{1\le j\le s}\{|(\gamma_j+1)_n|_v \}\le \|C_{\Lambda,\tau}(n)\|_v
  \le c_0 n^{\nu-1}|\tau|_v^{-n}\max_{1\le j\le s}\{\theta_j^{\nu_j-1}\}\max_{1\le j\le
  s}\{|(\gamma_j+1)_n|_v\}.
\end{aligned}
\end{equation}
Hence, putting $c_2=c_0 \max_{1\le j\le s}\{\theta_j^{\nu_j-1}\}$,
we get \eqref{eqlem1}, and therefore, by Lemma \ref{liouv}, the last
statement of Lemma \ref{lem1}.
\end{proof}
\begin{remark}\label{remin}
Notice that, for any matrix $Y \in\GL_\nu(K)$, we have
\begin{equation}
  \label{norm1}
\begin{aligned}
\|Y\|_v^{-1} \le \|Y^{-1}\|_v \le |\det Y|_v^{-1} \|Y\|_v^{\nu-1}.
\end{aligned}
\end{equation}
\end{remark}
Indeed, the relation ${\bI}_\nu=YY^{-1}$ implies  $1=\|{\bI}_\nu\|_v
\le \|Y\|_v \|Y^{-1}\|_v$  which gives the left inequality above.
The right inequality comes  from the  formula $Y^{-1}=(\det
Y)^{-1}\Adj(Y)$ where $\Adj(Y)$ denotes the adjoint of $Y$. Now, by
\eqref{norm1}, if $Z\in\GL_\nu(K)$, we have
\begin{equation}
  \label{norm2}
\begin{aligned}
\|Y\|_v|\det Z|_v\|Z\|_v^{1-\nu} \le \|Y\|_v\|Z^{-1}\|_v^{-1}\le
\|YZ\|_v\le \|Y\|_v\|Z\|_v.
\end{aligned}
\end{equation}
\begin{lemma}\label{lem2}
Let $\Lambda\in\GL_\nu(\bQ)$ be an invertible matrix such that all
its  eigenvalues $\gamma_1,\ldots,\gamma_s$ lie in
$\bQ\cap(\bZ_{p(v)}\setminus\bZ)$ $(s\le \nu)$. Let
$\nu_1\ldots,\nu_s$ be respectively the multiplicities  of
$\gamma_1,\ldots,\gamma_s$. Then there exist two positive real
numbers $c_3,c_4$ such that, for any positive integer $n>0$, we have
\begin{equation}
\label{eqlem2}
\begin{aligned}
&c_3(n+1)^{(1-\nu)^3}|\tau|_v^{n+2}\Big(\max_{1\le j\le
  s}\{|(-\gamma_j+1)_{n+1}|_v\}\Big)^{-(1-\nu)^2}\Big|\prod_{1\le j\le s}
\Big((-\gamma_j+1)_{n+1}\Big)^{\nu_j}\Big|_v^{\nu-2}\le\\
&
\|C_{\Lambda,\tau}(-n)\|_v\le\\
& c_4(n+1)^{(\nu-1)^2}|\tau|_v^{n+2}\Big(\max_{1\le j\le
s}\{|(-\gamma_j+1)_{n+1}|_v\}\Big)^{\nu-1}\Big|\prod_{1\le j\le s}
\Big((-\gamma_j+1)_{n+1}\Big)^{\nu_j}\Big|_v^{-1}.
\end{aligned}
\end{equation}
In particular,
$$\lim_{n\to +\infty}\|C_{\Lambda,\tau}(-n)\|_v^{1/n}=|\tau|_v\pi_{v}^{-1}.$$
\end{lemma}
\begin{proof}
Lemma \eqref{lem1} applies also for $-\Lambda$ instead of $\Lambda$
since the eigenvalues of $-\Lambda$ belong also to
$\bQ\cap(\bZ_{p(v)}\setminus\bZ)$. Then, there exist  two positive
real numbers $c_1',c_2'$,  such that, for any integer $n>0$, we have
\begin{equation}
\label{plem2}
\begin{aligned}
  c_1'|\tau|_v^{-(n+1)}\max_{1\le j\le s}\{|(-\gamma_j+1)_{n+1}|_v\}& \le \|C_{-\Lambda,\tau}(n+1)\|_v
  \le\\
  & c_2' (n+1)^{\nu-1}|\tau|_v^{-(n+1)}\max_{1\le j\le s}\{|(-\gamma_j+1)_{n+1}|_v\}.
\end{aligned}
\end{equation}
On the other hand, we have $\det C_{-\Lambda,\tau}(n+1)=
\tau^{-(n+1)\nu}\prod_{1\le j\le s}
\Big((-\gamma_j+1)_{n+1}\Big)^{\nu_j}$ (see proof of Lemma
\ref{lem1}), and by \eqref{invC},
\begin{equation*}
%\label{inv}
\begin{aligned} \Lambda^{-1} C_{-\Lambda,\tau}(n+1)^{-1}= \tau^{-1}C_{\Lambda,\tau}(-n).
\end{aligned}
\end{equation*}
Applying \eqref{norm2} and \eqref{norm1} with $Y=\Lambda^{-1}$ and
$Z=C_{-\Lambda,\tau}(n+1)^{-1}$, we get
\begin{equation}
\label{norm3}
\begin{aligned}
&\|\Lambda^{-1}\|_v\|C_{-\Lambda,\tau}(n+1)\|_v^{-(1-\nu)^2}|\det
C_{-\Lambda,\tau}(n+1)|_v^{\nu-2}\le
\|\Lambda^{-1}C_{-\Lambda,\tau}(n+1)^{-1}\|_v\le\\
&\hspace{7cm}\le
\|\Lambda^{-1}\|_v\|C_{-\Lambda,\tau}(n+1)\|_v^{\nu-1}|\det
C_{-\Lambda,\tau}(n+1)|_v^{-1}.
\end{aligned}
\end{equation}
Replacing now  $\det C_{-\Lambda,\tau}(n+1)$ (resp.
$\Lambda^{-1}C_{-\Lambda,\tau}(n+1)^{-1}$) with its value in
\eqref{norm3} and using \eqref{plem2},  we get
\begin{equation*}
%\label{norm4}
\begin{aligned}
&\|\Lambda^{-1}\|_v\Big(c_2' (n+1)^{\nu-1}\max_{1\le j\le
  s}\{|\tau^{-(n+1)}(-\gamma_j+1)_{n+1}|_v\}\Big)^{-(1-\nu)^2}\Big|\tau^{-(n+1)\nu}\prod_{1\le j\le s}
\Big((-\gamma_j+1)_{n+1}\Big)^{\nu_j}\Big|_v^{\nu-2}\\
&\le
\|\tau^{-1}C_{\Lambda,\tau}(-n)\|_v\le\\
& \|\Lambda^{-1}\|_v\Big(c_2' (n+1)^{\nu-1}\max_{1\le j\le
  s}\{|\tau^{-(n+1)}(-\gamma_j+1)_{n+1}|_v\}\Big)^{\nu-1}
  \Big|\tau^{-(n+1)\nu}\prod_{1\le j\le s} \Big((-\gamma_j+1)_{n+1}\Big)^{\nu_j}\Big|_v^{-1}.
\end{aligned}
\end{equation*}
or again
\begin{equation*}
%\label{norm4}
\begin{aligned}
&|\tau|_v^{n+2}\|\Lambda^{-1}\|_v\Big(c_2' (n+1)^{\nu-1}\max_{1\le
j\le
  s}\{|(-\gamma_j+1)_{n+1}|_v\}\Big)^{-(1-\nu)^2}\Big|\prod_{1\le j\le s}
\Big((-\gamma_j+1)_{n+1}\Big)^{\nu_j}\Big|_v^{\nu-2}\\
&\le
\|C_{\Lambda,\tau}(-n)\|_v\le\\
& |\tau|_v^{n+2} \|\Lambda^{-1}\|_v\Big(c_2' (n+1)^{\nu-1}\max_{1\le
j\le s}\{|(-\gamma_j+1)_{n+1}|_v\}\Big)^{\nu-1}\Big|\prod_{1\le j\le
s} \Big((-\gamma_j+1)_{n+1}\Big)^{\nu_j}\Big|_v^{-1}.
\end{aligned}
\end{equation*}
Putting now $c_3=\|\Lambda^{-1}\|_v(c_2')^{-(1-\nu)^2}$ and $c_4=
\|\Lambda^{-1}\|_v(c_2')^{\nu-1}$, we get \eqref{eqlem2}. The last
statement of Lemma \ref{lem2} follows from  \eqref{eqlem2}, Lemma
\ref{liouv}, and the fact: $\nu_1+\ldots\nu_s=\nu$.
\end{proof}
\subsection{4.3. Formal Laplace transform in several variables}
Let $\underline{\tau}=(\tau_1,\ldots,\tau_d)\in (K\setminus\{0\})^d$
and let $\underline{\Lambda}=(\Lambda_1,\ldots,\Lambda_d)\in
(\GL_\nu(K))^d$ such that the matrices  $\Lambda_i$ mutually commute
and such that all the eigenvalues of $\Lambda_1,\ldots,\Lambda_d$
are in $K\setminus \bZ$. Put
$$\underline{x}^{\underline{\Lambda}}:=x_1^{\Lambda_1}\ldots
x_d^{\Lambda_d},\quad \text{and}\quad \underline{\alpha}\bI_\nu
:=(\alpha_1\bI_\nu,\ldots,\alpha_d\bI_\nu).$$
\begin{definition}\label{def} \rm{For any integer
$\mu\ge 1$ and any $\mu\times \nu$ matrix $Y(\underline{x})=
\sum_{\underline{\alpha}\in \bN^d\cup(-\bN)^d}
Y_{\underline{\alpha}}\underline{x}^{\underline{\alpha}}$ with
entries in $K[[\underline{x},1/\underline{x}]]$, we define  the
\emph{Laplace transform of
$f:=Y(\underline{x})\underline{x}^{\underline{\Lambda}}$} with
respect to $\underline{\Lambda}$  and to $\underline{\tau}$ as
follows:}
\begin{equation}
\label{trf2}
\begin{aligned}
{\mathcal
L}^{\underline{\tau}}_{\underline{\Lambda}}(Y(\underline{x})\underline{x}^{\underline{\Lambda}})
=&\sum_{\underline{\alpha}\in \bN^d\cup(-\bN)^d}
Y_{\underline{\alpha}}\prod_{1\le i\le d}{\mathcal
L}^{\tau_i}_{\Lambda_i}(x_i^{\Lambda_i+\alpha_i\bI_\nu})=\sum_{\underline{\alpha}\in
\bN^d\cup(-\bN)^d} Y_{\underline{\alpha}}\prod_{1\le i\le
d}C_{\Lambda_i,\tau_i}(\alpha_i)x_i^{-\Lambda_i-(\alpha_i+1)\bI_\nu}\\
=&\sum_{\underline{\alpha}\in \bN^d\cup(-\bN)^d}
Y_{\underline{\alpha}}\prod_{1\le i\le
d}C_{\Lambda_i,\tau_i}(\alpha_i)\underline{x}^{-\underline{\Lambda}-(\underline{\alpha}+\underline{1})\bI_\nu}\\
=&\Big(\sum_{\underline{\alpha}\in \bN^d\cup(-\bN)^d}
Y_{\underline{\alpha}}\prod_{1\le i\le
d}C_{\Lambda_i,\tau_i}(\alpha_i)\underline{x}^{-\underline{\alpha}\bI_\nu}\Big)
\underline{x}^{-\underline{\Lambda}-\bI_\nu},
\end{aligned}
\end{equation}
and we set
$$Z^{\underline{\tau}}_{\underline{\Lambda},\underline{\alpha}}=Y_{-\underline{\alpha}}\prod_{1\le
i\le d}C_{\Lambda_i,\tau_i}(-\alpha_i),\quad \text{for all}
\quad\underline{\alpha}\in \bN^d\cup(-\bN)^d,$$ and,
$$Z^{\underline{\tau}}_{\underline{\Lambda}}(\underline{x})
=\sum_{\underline{\alpha}\in \bN^d\cup(-\bN)^d}
Z^{\underline{\tau}}_{\underline{\Lambda},\underline{\alpha}}\underline{x}^{-\underline{\alpha}},$$ so that
$${\mathcal
L}^{\underline{\tau}}_{\underline{\Lambda}}(Y(\underline{x})\underline{x}^{\underline{\Lambda}})
=Z^{\underline{\tau}}_{\underline{\Lambda}}(\underline{x})\underline{x}^{-\underline{\Lambda}-\bI_\nu}.$$
\end{definition}
\begin{remark}\rm{The transformation ${\mathcal
L}^{\underline{\tau}}_{\underline{\Lambda}}$ is well defined since,
by construction, all the matrices $C_{\Lambda_i,\tau_i}(\alpha_i)$
mutually commute  for all $\underline{\alpha}$. In addition,
according to formula \eqref{Ld=xL}, we check easily that this
Laplace transform commutes with the derivations in the following
sense:
\begin{equation}
\label{op3} \quad {\mathcal
L}^{\underline{\tau}}_{\underline{\Lambda}}\Big(\underline{\partial}^{\underline{\beta}}(f)\Big)
= \prod_{1\le i\le d}(\tau_ix_i)^{\beta_i} {\mathcal
L}^{\underline{\tau}}_{\underline{\Lambda}}(f)\;\; \text{and}\;\;
{\mathcal L}^{\underline{\tau}}_{\underline{\Lambda}}(\prod_{1\le
i\le d}x_i^{\beta_i}f) =
\frac{(-1)^{|\underline{\beta}|}}{\prod_{1\le i\le
d}\tau_i^{\beta_i}}\underline{\partial}^{\underline{\beta}}({\mathcal
L}^{\underline{\tau}}_{\underline{\Lambda}}(f)),
\end{equation}
for any $\underline{\beta}=(\beta_1\ldots,\beta_d)\in\bN^d$. Also,
we find}
\begin{equation}
\label{op3'} {\mathcal
L}^{\underline{\tau}}_{-\underline{\Lambda}}\big({\mathcal
L}^{\underline{\tau}}_{\underline{\Lambda}}(Y(\underline{x})\underline{x}^{\underline{\Lambda}})\big)
= (-1)^d \prod_{1\le i\le d}\tau_i Y(-\underline{x})\prod_{1\le i\le
d}\Lambda_i^{-1} x^{\underline{\Lambda}},\quad\text{where}\quad
-\underline{x}=(-x_1,\ldots,-x_d).
\end{equation}
\end{remark}
This leads to state
\begin{proposition}\label{pro2} Assume that $\mu=1$ and  that all the
entries of $f$ are solutions of a differential equation $\phi\in
K[x_1,\ldots,x_d,\partial_1,\ldots,\partial_d]$. Then, all the
entries of ${\mathcal
L}^{\underline{\tau}}_{\underline{\Lambda}}(f)$ are solutions of
$\mathcal{F}_{\underline{\tau}}(\phi)$.
\end{proposition}
The formal transformation ${\mathcal
L}^{\underline{\tau}}_{\underline{\Lambda}}$ has moreover the
following arithmetic properties:
\begin{proposition}\label{pro3}  Let $v$ be a
finite place in $\Sigma_{\rm f}$. Under notation of definition \ref{def},
assume that the matrices $\Lambda_1,\ldots,\Lambda_d$ belong to $
\GL_\nu(\bQ)$, all their eigenvalues are in
$\bQ\cap\bZ_{p(v)}\setminus\bZ$ and
$\tau_1=\tau_2\ldots=\tau_d=\tau\in K\setminus\{0\}$. Then
$$\disp\limsup_{|\underline{\alpha}|\to+\infty}
\|Z^{\underline{\tau}}_{\underline{\Lambda},\underline{\alpha}}\|_v^{1/|\underline{\alpha}|}
\le \pi_{v}^{-1}|\tau|_v\limsup_{|\underline{\alpha}|\to+\infty}
\|Y_{-\underline{\alpha}}\|_v^{1/|\underline{\alpha}|},$$and
$$\disp\limsup_{|\underline{\alpha}|\to+\infty}
\|Z^{\underline{\tau}}_{\underline{\Lambda},-\underline{\alpha}}\|_v^{1/|\underline{\alpha}|}
\le \pi_{v}|\tau|_v^{-1}\limsup_{|\underline{\alpha}|\to+\infty}
\|Y_{\underline{\alpha}}\|_v^{1/|\underline{\alpha}|}.$$
\begin{proof}
Let $\gamma_{i,1},\ldots,\gamma_{i,\nu}$ be the eigenvalues of
$\Lambda_i$ for $i=1,\ldots,d$. According to lemmas \ref{lem1} and
\ref{lem2}, for each $C_{\Lambda_i,\tau}$, there exist two positive
constants $c_{2,i},c_{4,i}>0$ such that for any $\alpha_i
>0$, we have
$$\|C_{\Lambda_i,\tau}(\alpha_i)\|_v
  \le c_{2,i} \alpha_i^{\nu-1}|\tau|_v^{-\alpha_i}\max_{1\le j\le
  \nu}\{|(\gamma_{i,j}+1)_{\alpha_i}|_v\}.$$
  and
\begin{equation*}
\begin{aligned}
\|C_{\Lambda_i,\tau}(-\alpha_i)\|_v\le\\
&
c_{4,i}(\alpha_i+1)^{(\nu-1)^2}|\tau|_v^{\alpha_i+2}\Big(\max_{1\le
j\le
\nu}\{|(-\gamma_{i,j}+1)_{\alpha_i+1}|_v\}\Big)^{\nu-1}\Big|\prod_{1\le
j\le \nu} (-\gamma_{i,j}+1)_{\alpha_i+1}\Big|_v^{-1}.
\end{aligned}
\end{equation*}
If we set $\underline{j}=(j_1,\ldots,j_d)$, we get
$$\prod_{1\le i\le  d}\|C_{\Lambda_i,\tau}(\alpha_i)\|_v
  \le \prod_{1\le i\le  d}(c_{2,i} \alpha_i^{\nu-1})|\tau|_v^{-|\underline{\alpha}|}
  \max_{\underline{j}\le (\nu,\ldots,\nu)}\Big\{\Big|\prod_{1\le i\le  d}(\gamma_{i,j_i}+1)_{\alpha_i}\Big|_v\Big\}.$$
  and
\begin{equation*}
\begin{aligned}
&\prod_{1\le i\le  d}\|C_{\Lambda_i,\tau}(-\alpha_i)\|_v\le\\
& \prod_{1\le i\le
d}(c_{4,i}(\alpha_i+1)^{(\nu-1)^2})|\tau|_v^{|\underline{\alpha}|+2d}
\Big(\max_{\underline{j}\le (\nu,\ldots,\nu)}\Big\{\Big|\prod_{1\le
i\le d}
(-\gamma_{i,j_i}+1)_{\alpha_i+1}\Big|_v\Big\}\Big)^{\nu-1}\Big|\prod_{1\le
i\le d\atop 1\le j\le \nu}
(-\gamma_{i,j}+1)_{\alpha_i+1}\Big|_v^{-1}.
\end{aligned}
\end{equation*}
In addition, since all the eigenvalues $\gamma_{i,j_i}$ lie in
$\bQ\cap\bZ_{p(v)}\setminus\bZ$, we have by Lemma \ref{liouv},
$$\limsup_{|\underline{\alpha}|\longrightarrow \infty}\Big|\prod_{1\le
i\le d}
(-\gamma_{i,j_i}+1)_{\alpha_i+1}\Big|_v^{1/\underline{\alpha}}=
\limsup_{|\underline{\alpha}|\longrightarrow \infty}\Big|\prod_{1\le
i\le d}
(-\gamma_{i,j_i}+1)_{\alpha_i}\Big|_v^{1/\underline{\alpha}}=\pi_v.$$
Combining these observations  with the fact
$\|Z^{\underline{\tau}}_{\underline{\Lambda},\underline{\alpha}}\|_v
\le \|Y_{-\underline{\alpha}}\|_v \prod_{1\le i\le
d}\|C_{\Lambda_{i},{\tau}_{i}}(-\alpha_i)\|_v$ for any
$\underline{\alpha}\in\bN^d\cup(-\bN)^d$, we get the proposition.
\end{proof}
\end{proposition}

\end{document}